\documentclass[12pt]{amsart}
\usepackage{amsmath,amsthm,amsfonts,latexsym,amssymb,amscd,color}
\usepackage{graphics,multimedia,tikz,graphicx}
\usepackage{url}
\usepackage{enumitem}
\numberwithin{equation}{section}
\theoremstyle{plain} 
\newtheorem{theorem}{Theorem}[section]

\newtheorem{corollary}[theorem]{Corollary}
\newtheorem{claim}[theorem]{Claim}
\newtheorem{subclaim}[theorem]{Subclaim}
\theoremstyle{definition}

\newcommand{\wec}[1]{{\mathbf{#1}}}
\newcommand{\wrel}[1]{\;#1\;}

\newcommand{\Con}{{\mathrm{Con}}}

\begin{document}


\title[Subalgebras as retracts of finite subdirect powers]{Representing subalgebras as retracts \\
  of finite subdirect powers}

\author[K. A. Kearnes]{Keith A. Kearnes}
\address{Department of Mathematics\\
University of Colorado\\
Boulder, CO 80309-0395\\
USA}
\email{kearnes@colorado.edu}

\author[A. Rasstrigin]{Alexander Rasstrigin}
\address{Department of Higher Mathematics and Physics\\
Volgograd State Socio-Pedagogical University\\
Volgograd\\
Russia}
\email{rasal@fizmat.vspu.ru}

\thanks{This material is based upon work supported by
  the National Science Foundation grant no.\ DMS 1500254.}



\begin{abstract}
  We prove that if   $\mathbb A$ is an algebra that is 
  supernilpotent with respect to the $2$-term
  higher commutator, and $\mathbb B$ is a subalgebra
  of $\mathbb A$, then $\mathbb B$
  is representable as a
  retract
  of a finite subdirect power of $\mathbb A$.
\end{abstract}

\maketitle


\section{Introduction}\label{intro}
The paper \cite{neumann} by Peter M. Neumann
opens with the remark that, in conversations
with other group theorists, the following question
arose several times: Must a formation of finite
nilpotent groups be closed under subgroups?
Since the formation generated by a finite group
$\mathbb A$ is the class of groups that are representable
as homomorphic images
of finite subdirect powers of $\mathbb A$,
this question is equivalent to the following
one: If $\mathbb A$ is a finite nilpotent group
and $\mathbb B$ is a subgroup of $\mathbb A$, must $\mathbb B$
be representable as a homomorphic image
of a finite subdirect power of $\mathbb A$?

Neumann proves that the answer is Yes.
His proof relies on the
nonobvious result of Michael Vaughan-Lee
that any characteristic subgroup
of a relatively free class-$c$ nilpotent group of
rank $k$ must be fully invariant, provided $k>c$
(cf.\ \cite{vaughan-lee}).
%

Using different ideas, 
our paper extends Neumann's result from finite groups
to arbitrary algebraic structures.
Where Neumann focuses on finite groups that are nilpotent,
we focus instead on possibly-infinite
algebras that are supernilpotent with respect to the $2$-term
higher commutator. Our theorem includes his if
one constrains its scope to finite groups.

We close the paper by describing an 
algebra $\mathbb A$ which shows that our
result is, in a sense, sharp.
Our algebra has the following properties:
\begin{enumerate}
\item  $\mathbb A$ is an expansion of a finite group.
\item  $\mathbb A$ is $2$-step nilpotent as an algebraic structure,
  but it is not supernilpotent.
\item  $\mathbb A$ has a subalgebra $\mathbb B$ that is not
representable as a homomorphic image
of a finite subdirect power of $\mathbb A$.
\end{enumerate}

\section{Discussion}

Our main theorem is

\begin{theorem} \label{main}
If $\mathbb A$ is an algebra that is 
supernilpotent with respect to the $2$-term
higher commutator, and $\mathbb B$ is a subalgebra
of $\mathbb A$, then $\mathbb B$
is representable as a retract
of a finite subdirect power of $\mathbb A$.
\end{theorem}

In the Introduction we described the result in a slightly weaker form,
namely that

\begin{theorem} \label{main-x}
If $\mathbb A$ is an algebra that is 
supernilpotent with respect to the $2$-term
higher commutator, and $\mathbb B$ is a subalgebra
of $\mathbb A$, then $\mathbb B$
is representable as a \emph{quotient}
of a finite subdirect power of $\mathbb A$.
\end{theorem}

The difference in the two forms is that $\mathbb B$
is a \emph{quotient} of $\mathbb D$ if there is
a surjective homomorphism $\mu: \mathbb D\to \mathbb B$,
while $\mathbb B$
is a \emph{retract} of $\mathbb D$ if there is
a surjective homomorphism $\mu: \mathbb D\to \mathbb B$
which has a homomorphism
$\nu:\mathbb B\to \mathbb D$ that is a right inverse
($\mu\circ\nu=\textrm{id}_{\mathbb B}$).

The main result can be phrased as a theorem about formations of algebras.

\subsection{\sc Classes of algebras.}

We use the following terminology for classes
of algebras of the same type.
\begin{enumerate}
\item A \emph{variety} is a class of algebras
  definable by a set of identities.
\item (Provisional) A \emph{pseudovariety} is a class of finite algebras
closed under
the construction of
homomorphic images,
subalgebras, and finite products.
\item (Provisional) A \emph{formation} is a class of finite algebras
closed under
the construction of
homomorphic images
and finite subdirect products.
\item A class of algebras is \emph{hereditary}
if it is closed under
subalgebras.
\item A class of algebras is \emph{axiomatic}
if it first-order axiomatizable.
\end{enumerate}
In this paper we shall allow
all classes to contain infinite algebras,
so we set aside the two provisional definitions
above (formation, pseudovariety)
and adopt these definitions instead:
a \emph{formation} is \emph{any} class of similar
algebras 
closed under
the construction of
homomorphic images
and finite subdirect products, while a pseudovariety
is \emph{any} class of similar algebras 
closed under
the construction of
homomorphic images, subalgebras, and finite products.

Under either set of definitions, a class is a pseudovariety
if and only if it is a hereditary formation.
Moreover, under the adopted definitions, we have:

\begin{theorem}[Theorem 1.1 \cite{formations}]
  Let $\mathcal V$
  be a class of similar algebraic structures. The following conditions
are equivalent:
\begin{enumerate}
\item $\mathcal V$ is a variety.
\item $\mathcal V$ is closed under homomorphic images, subalgebras and products.
\item $\mathcal V$ is closed under homomorphic images and subdirect products.
\item $\mathcal V$ is closed under homomorphic images,
  subalgebras, finite products,
  and $\mathcal V$ is axiomatizable.
\item $\mathcal V$
  is an axiomatic formation that is closed under subalgebras.
\end{enumerate}
\end{theorem}

Hence $\mathcal V$ is a variety if and only if it is an axiomatizable
pseudovariety if and only if it is an axiomatizable hereditary formation.

It is not hard to produce formations that are
not hereditary, that is, are not closed under the construction
of subalgebras. For example the formation generated
by the alternating group $A_5$ consists of the groups
isomorphic to finite powers
of $A_5$, so the proper nontrivial subgroups of $A_5$ are
in the pseudovariety generated by $A_5$ but not
in the formation generated by $A_5$.

Using the terminology of formations,
the most obvious corollary of our main theorem
is the following:

\begin{corollary}
  \label{main1}
Let $\mathcal F$ be a formation of algebras whose members are
supernilpotent in the sense of the $2$-term higher commutator.
Then
\begin{enumerate}
\item $\mathcal F$ is hereditary. (Equivalently, $\mathcal F$ is a pseudovariety.)
\item $\mathcal F$ is a variety if and only if $\mathcal F$ is axiomatic.
\end{enumerate}
\end{corollary}

Another easily-derived
corollary of Theorem~\ref{main}
which is best stated in terms of classes of algebras
is the following:

\begin{corollary}
  \label{main1.5}
Let $\mathcal V$ be a variety of algebras whose members are
supernilpotent in the sense of the $2$-term higher commutator.
If $\mathbb X\in \mathcal V$ is finite, then any
finite member of the variety
generated by $\mathbb X$ is representable as a
quotient of a finite subdirect power of $\mathbb X$.
\end{corollary}

\noindent
[Idea of proof:] Any finite algebra
$\mathbb Y\in {\mathcal V}(\mathbb X)$
is a quotient of a finite, relatively free algebra
$\mathbb F_{{\mathcal V}(\mathbb X)}(|\mathbb Y|)$.
Since $\mathbb F_{{\mathcal V}(\mathbb X)}(|\mathbb Y|)$
is isomorphic to a subalgebra of
$\mathbb X^{|\mathbb X|^{|\mathbb Y|}}$, we can 
apply Theorem~\ref{main-x} to
$\mathbb A = \mathbb X^{|\mathbb X|^{|\mathbb Y|}}$
and
$\mathbb B = \mathbb F_{{\mathcal V}(\mathbb X)}(|\mathbb Y|)$
to obtain that
$\mathbb F_{{\mathcal V}(\mathbb X)}(|\mathbb Y|)$
is isomorphic to a quotient of a finite
subdirect power of $\mathbb X$.
By taking a further quotient, we get that
$\mathbb Y$ is isomorphic to a quotient of a finite
subdirect power of $\mathbb X$.

\bigskip
  
We shall derive Theorem~\ref{main}
from a slightly more general statement,
Theorem~\ref{base}. Before proving that
result we clarify some of the language
used in Theorem~\ref{main}.

\subsection{\sc $2$-term commutator versus ordinary commutator.}

The theory of the ordinary $1$-term binary
commutator, $[\alpha,\beta]$,
is developed for arbitrary algebraic
structures in Chapter 3 of \cite{hobby-mckenzie}
by Hobby and McKenzie.
The $2$-term binary commutator, $[\alpha,\beta]_2$, was
introduced by Emil Kiss in \cite{kiss}.
Kiss proves in \cite{kiss} that
the $2$-term binary commutator is larger or equal to the
ordinary binary commutator ($[\alpha,\beta]\leq [\alpha,\beta]_2$),
and that the two are equal on any algebra in
a congruence modular variety. In
\cite{kearnes-szendrei}, the equality
of these and other commutators is proved to hold
for any variety satisfying
a nontrivial idempotent Maltsev condition.
These concepts ($2$-term commutator versus ordinary commutator)
for the binary commutator
have analogues for higher arity commutators. The equality of the
ordinary higher commutator and the
$2$-term higher commutator for
congruence modular varieties was proved
by Andrew Moorhead in \cite[Theorem~6.4]{moorhead}.
Moorhead has recently extended this result to
show (in \cite{moorhead_arxiv}) that the
ordinary higher commutator agrees with the 
$2$-term higher commutator in any variety satisfying a
nontrivial idempotent Maltsev condition.
Hence we can rephrase Theorem~\ref{main} to eliminate
the ``$2$-term'' part, provided we restrict
the scope of the theorem
to algebras satisfying a nontrivial idempotent
Maltsev condition. The theorem would then read:

\begin{corollary} \label{main2}
  Assume that $\mathbb A$ is an algebra
  satisfying a nontrivial idempotent Maltsev condition.
If $\mathbb A$ 
is supernilpotent, and $\mathbb B$ is a subalgebra
of $\mathbb A$, then $\mathbb B$
is representable as a retract
of a finite subdirect power of $\mathbb A$.

Hence, if $\mathcal F$ is a formation of supernilpotent algebras
in which every member satisfies a nontrivial idempotent
Maltsev condition, then $\mathcal F$ is a pseudovariety.
\end{corollary}

\subsection{\sc Supernilpotence versus ordinary nilpotence.}
Supernilpotence is a form of nilpotence defined
in terms of the higher commutator.
It is proved in \cite{kearnes-szendrei-super}
that supernilpotence implies nilpotence for finite algebras.
From the results of \cite{aichinger-mudrinski,spec},
the exact degree to which supernilpotence
is stronger than nilpotence is well understood
for finite algebras satisfying nontrivial idempotent
Maltsev conditions.
This allows us to rewrite Corollary~\ref{main2}
in a way that refers to nilpotence instead
of supernilpotence:

\begin{corollary}  \label{main3}
Assume that $\mathbb A$ is a finite algebra
satisfying a nontrivial idempotent Maltsev condition.
If $\mathbb A$ is a product of nilpotent algebras
of prime power cardinality and $\mathbb B$ is a subalgebra
of $\mathbb A$, then $\mathbb B$
is representable as a retract
of a finite subdirect power of $\mathbb A$.

Hence, if $\mathcal F$ is a formation of finite nilpotent algebras
in which every
member satisfies a nontrivial idempotent
Maltsev condition, and every directly indecomposable member
of $\mathcal F$ has prime power cardinality,
then $\mathcal F$ is a pseudovariety.
\end{corollary}

This consequence of Theorem~\ref{main}
includes Neumann's result, since (i) any group satisfies a nontrivial
idempotent Maltsev condition, and
(ii) a finite,
directly indecomposable, nilpotent group has prime power cardinality.

\subsection{\sc When $\mathbb A$ satisfies no nontrivial
      idempotent Maltsev condition.}
It would be wrong to treat Corollaries~\ref{main2} and \ref{main3}
as if they expressed the essential content of Theorem~\ref{main}.
If $\mathbb A$ is a finite nilpotent algebra,
and $\mathbb A$ satisfies a nontrivial idempotent Maltsev condition, then
it can be shown that
$\mathbb A$ must have a Maltsev term operation
$p(x,y,z)$, i.e. a term operation for which
$\mathbb A\models p(x,x,y)\approx y\approx p(y,x,x)$.
By Corollary~7.4 of \cite{freese-mckenzie},
the nilpotence of $\mathbb A$ forces
$p$ to be invertible in its first and last variables,
so this term operation behaves much like the group term operation
$p(x,y,z)=xy^{-1}z$ which controls 
most of a group's properties. (The term operations
of any group are generated by $p(x,y,z)=xy^{-1}z$
along with the constant $1$.)
In this circumstance
$\mathbb A$ is a ``group-like'' algebra.
In this context, 
Corollaries~\ref{main2}, \ref{main3} are only slight
extensions of what was already known.
The real interest in Theorem~\ref{main}
should be in what it says about formations
with members that do not satisfy
any nontrivial idempotent Maltsev condition,
as in the following corollary.

\begin{corollary}  \label{main4}
  Assume that $\mathbb A$ is an algebra
which has a finite bound on the essential
arity of its term operations.
If $\mathbb B$ is a subalgebra
of $\mathbb A$, then $\mathbb B$
is representable as a retract
of a finite subdirect power of $\mathbb A$.

Hence, if $\mathcal F$ is a formation of algebras
and every member
of $\mathcal F$ has a finite bound on the essential
arity of its term operations, 
then $\mathcal F$ is a pseudovariety.
\end{corollary}

This corollary is a consequence of Theorem~\ref{main},
since any algebra of essential arity at most $k$
will be supernilpotent of class at most $k$ in the sense of the
$2$-term higher commutator.

Even the following consequence of Corollary~\ref{main4}
seems to be new:

\begin{corollary}  \label{main5}
Any formation of unary algebras is a pseudovariety.
\end{corollary}

\section{The proof of Theorem~\ref{main}}

We begin this section by defining what it means for
a congruence to be supernilpotent of class $c=2$
with respect to the $2$-term higher commutator,
and then we indicate how to modify the definition
for other values of $c$.

Let $\theta$ be a congruence on an algebra $\mathbb A$.
We shall define a subalgebra $M(\theta,\theta,\theta)$
of $\mathbb A^{2^3}$ by indicating its generating $2^3$-tuples.
The \emph{coordinates} of a $2^3$-tuple are the elements of the set
$2^3=\{0,1\}^3$. Since a coordinate of a tuple $t\in\mathbb A^{2^3}$
is itself a tuple $(a,b,c)\in 2^3$,
and since we dread the prospect of referring
to ``coordinates of coordinates'', we shall refer
to an element $(a,b,c)\in 2^3$ as an ``address of a coordinate''
if we are treating it as a string of $0$'s and $1$'s,
and we wish to refer to different places in the string.

We arrange the coordinates in the shape of a $3$-dimensional cube
by stipulating that two coordinates are adjacent if the
Hamming distance of their addresses is $1$.
See Figure~\ref{eight_coords}.

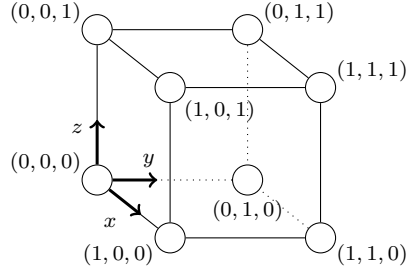
\begin{figure}[!htbp]
\begin{tikzpicture}[z={(4.85mm,-3.85mm)},every node/.style={minimum size=0.4cm,inner sep=0cm,font=\tiny},every circle node/.style={draw}]
  \node [circle,label={175:$(0,0,0)$}] (a) at (0,0,0) {};
  \node [circle,label={175:$(0,0,1)$}] (b) at (0,2,0) {} edge (a);
  \node [circle,label={270:$(0,1,0)$}] (c) at (2,0,0) {} edge [dotted] (a);
  \node [circle,label={5:$(0,1,1)$}] (d) at (2,2,0) {} edge (b) edge [dotted] (c);
  \node [circle,label={185:$(1,0,0)$}] (e) at (0,0,2) {} edge (a);
  \node [circle,label={330:$(1,0,1)$}] (f) at (0,2,2) {} edge (b) edge (e);
  \node [circle,label={355:$(1,1,0)$}] (g) at (2,0,2) {} edge [dotted] (c) edge (e);
  \node [circle,label={5:$(1,1,1)$}] (h) at (2,2,2) {} edge (d) edge (f) edge (g);
  \draw [->,line width=0.4mm] (a) -- (0,0,1.2);
  \draw [->,line width=0.4mm] (a) -- (0,0.8,0);
  \draw [->,line width=0.4mm] (a) -- (0.8,0,0);
  \node [label={185:$x$},minimum size=0.1cm] (x) at (0,0,0.9) {};
  \node [label={95:$y$},minimum size=0.1cm] (y) at (0.9,0,0) {};
  \node [label={183:$z$},minimum size=0.1cm] (z) at (0,0.9,0) {};
\end{tikzpicture}
\bigskip
\begin{center}
\caption{\sc Eight coordinates, arranged by adjacency of address.} \label{eight_coords}
\end{center}
\end{figure}

We have labeled the first, second, and third
spatial axes of the cube
with the letters $x$, $y$, and $z$.
If we move from one address to another in the
$x$-direction, then the address changes only in its
first place, i.e., addresses of the form
$(0,b,c)$ change to $(1,b,c)$. Similarly, 
if we move from one address to another in the $y$- or
$z$-directions, then the address changes only in its
second or third places. The \emph{last coordinate} among all
$2^3$ coordinates will be the coordinate with address $(1,1,1)$.
The \emph{earlier coordinates} will be all coordinates that are not the last
coordinate.

The \emph{standard generators of $M(\theta,\theta,\theta)$
in the $x$-direction} are all tuples in $\mathbb A^{2^3}$
which have the form indicated in Figure~\ref{standard_gen},
where $(u,v)\in\theta$.
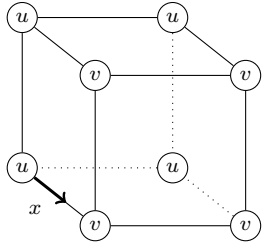
\begin{figure}[!htbp]
\begin{tikzpicture}[z={(4.85mm,-3.85mm)},every node/.style={minimum size=0.4cm,inner sep=0cm,font=\tiny},every circle node/.style={draw}]
  \node [circle] (a) at (0,0,0) {$u$};
  \node [circle] (b) at (0,2,0) {$u$} edge (a);
  \node [circle] (c) at (2,0,0) {$u$} edge [dotted] (a);
  \node [circle] (d) at (2,2,0) {$u$} edge (b) edge [dotted] (c);
  \node [circle] (e) at (0,0,2) {$v$} edge (a);
  \node [circle] (f) at (0,2,2) {$v$} edge (b) edge (e);
  \node [circle] (g) at (2,0,2) {$v$} edge [dotted] (c) edge (e);
  \node [circle] (h) at (2,2,2) {$v$} edge (d) edge (f) edge (g);
  \draw [->,line width=0.4mm] (a) -- (0,0,1.2);
  \node [label={185:$x$},minimum size=0.1cm] (x) at (0,0,0.9) {};
\end{tikzpicture}
\begin{center}
\caption{\sc A standard generator in the $x$-direction.}  \label{standard_gen}
\end{center}
\end{figure}
That is, for each pair $(u,v)\in\theta$
there is a standard generator of $M(\theta,\theta,\theta)$
in the $x$-direction,
which is a $2^3$-tuple $t$, where at every address $\sigma = (x,y,z)$
with $x=0$ we have coordinate value $t_{\sigma}=u$ and at every address
$\sigma$ with $x=1$ we have coordinate value $t_{\sigma}=v$.
(Visually, $t_{\sigma} = u$ on the 
``\emph{$x=0$ hyperface}'' of the cube, while
$t_{\sigma} = v$ on the 
``\emph{$x=1$ hyperface}''.)
Define standard generators in the $y$- and $z$-directions
analogously.

$M(\theta,\theta,\theta)$ is defined to be the subalgebra
of $\mathbb A^{2^3}$ that is generated by the standard
generators in all directions. This subalgebra contains the diagonal
of $\mathbb A^{2^3}$, since the constant tuples ( = diagonal elements)
of $\mathbb A^{2^3}$
are standard
generators in each direction.

A congruence $\theta$ is called \emph{supernilpotent}
(of class $2$) with respect to the $2$-term higher commutator
if whenever two elements $s, t\in M(\theta,\theta,\theta)$
agree at all earlier coordinates, then they also
agree at the last coordinate.
This can be expressed in many different ways, e.g.:
\begin{enumerate}
\item (The definition)
  If $s, t\in M(\theta,\theta,\theta)$ and
$s_{\sigma}=t_{\sigma}$ for all $\sigma\in\{0,1\}^3$
other than $\sigma = (1,1,1)$, then we also have
$s_{(1,1,1)}=t_{(1,1,1)}$.
\item $M(\theta,\theta,\theta)$ is the graph
  of a function whose domain is the projection
  of $M(\theta,\theta,\theta)$ onto
  the earlier coordinates and whose range is
  the projection of $M(\theta,\theta,\theta)$ onto
the last coordinate.
\item The projection of $M(\theta,\theta,\theta)$
  onto its seven earlier coordinates is a bijection.
\item (Since $M(\theta,\theta,\theta)$ is a subalgebra
  of $\mathbb A^{2^3}$):
  $M(\theta,\theta,\theta)$ is the graph of a partial
homomorphism $\mu: \mathbb A^7\to \mathbb A$.
I.e., $M(\theta,\theta,\theta)$ is the graph
of a homomorphism  $\mu: \mathbb D\to \mathbb A$
for $\mathbb D\;(\leq \mathbb A^7)$ equal to the projection
of $\mathbb A^{2^3}$ to its seven earlier coordinates.
\end{enumerate}

The only difference between supernilpotence of class $2$
and supernilpotence of class $c$ is that for class
$c$ we use $(c+1)$-dimensional hypercubes in place of
$3$-dimensional cubes. The vertices of these cubes
have addresses in $\{0,1\}^{c+1}$,
$x_1$- through $x_{c+1}$-spatial directions,
$M(\theta,\ldots,\theta)$ is generated by standard
generators in all of the directions, and
$\theta$ is supernilpotent of class $c$ with respect
to the $2$-term higher commutator if the last
coordinate of any tuple in
$M(\theta,\ldots,\theta)$ depends functionally
on the earlier coordinates.

For more about supernilpotence and higher commutator
theory see \cite{aichinger-mudrinski,moorhead}.

We need one more definition before proceeding.
If $\mathbb A$ is an algebra,
$B\subseteq A$ is a subset,
and $\theta\in\Con(\mathbb A)$ is a congruence,
then
\[
B^{\theta} = \bigcup_{b\in B} b/\theta
= \{a\in A \mid (a,b)\in\theta \text{ for some } b \in B \}
  \]
  is the \emph{saturation} of
  $B$ by $\theta$.
The saturation of $B$ by $\theta$
is the least subset of
  $\mathbb A$ containing $B$ that is a union
  of $\theta$-classes.
  If $B$ is a subuniverse of $\mathbb A$,
  then it can be shown that $B^{\theta}$ is also a subuniverse.
  
\begin{theorem} \label{base}
Let $\mathbb A$ be an algebra,
and let $\mathbb B$ be a subalgebra of $\mathbb A$ and
$\theta\in\Con(\mathbb A)$.
If
\begin{enumerate}
\item $\mathbb B^{\theta} = \mathbb A$, 
\item $\theta$ is supernilpotent with respect
to the $2$-term higher commutator,
\end{enumerate}
then $\mathbb B$ is a retract of a finite subdirect
power of $\mathbb A$.
\end{theorem}

\begin{proof}
  We shall draw pictures as if we working in three
  spatial dimensions, but it will be clear that the arguments
  we give work in dimension $c+1$ for any
  supernilpotence class $c$.

Let $\Gamma$ be the subset of the standard generators of
$M(\theta,\ldots,\theta)$ consisting of
only those standard generators (in all directions)
where the value in the last coordinate
lies in the subalgebra $\mathbb B$, as indicated in Figure~\ref{Gamma}.

\begin{figure}[!htbp]
\begin{tikzpicture}[z={(4.85mm,-3.85mm)},every node/.style={minimum size=0.4cm,inner sep=0cm,font=\tiny},every circle node/.style={draw}]
  \node [circle] (a) at (0,0,0) {};
  \node [circle] (b) at (0,2,0) {} edge (a);
  \node [circle] (c) at (2,0,0) {} edge [dotted] (a);
  \node [circle] (d) at (2,2,0) {} edge (b) edge [dotted] (c);
  \node [circle] (e) at (0,0,2) {} edge (a);
  \node [circle] (f) at (0,2,2) {} edge (b) edge (e);
  \node [circle] (g) at (2,0,2) {} edge [dotted] (c) edge (e);
  \node [circle,label={0:$~\in B$}] (h) at (2,2,2) {$b$} edge (d) edge (f) edge (g);
\end{tikzpicture}

\begin{center}
\caption{\sc A typical element of $\Gamma$.}\label{Gamma}  
\end{center}
\end{figure}
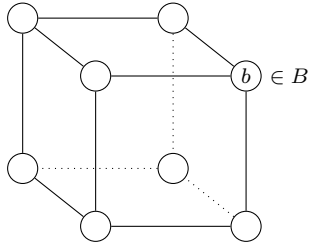

\begin{claim} \label{base_claim}
\mbox{}
  
\begin{enumerate}
\item Every tuple in $\Gamma$ has last coordinate in $\mathbb B$.
\item For any $b\in B$, $\Gamma$ contains a tuple 
  whose last coordinate is $b$.
\item For any $a\in A$ and any earlier coordinate $\sigma$,
  $\Gamma$ contains a tuple $t$ such that
  $t_{\sigma}=a$.
\end{enumerate}
\end{claim}

Item (1) of Claim~\ref{base_claim} is part of the
definition of $\Gamma$.

Item (2) of Claim~\ref{base_claim} follows from the fact
that $\Gamma$ contains all of the diagonal tuples
with diagonal value in $\mathbb B$.

For Item (3) of Claim~\ref{base_claim}, fix any $a\in A$
and any coordinate $\sigma$ other than the last coordinate.
Since $\mathbb B^{\theta} = \mathbb A$, there exists some
$b\in \mathbb B$ such that $(a,b)\in\theta$.
We explain why there is a tuple $t\in \Gamma$
which has $b$ in the last coordinate and $a$
in coordinate $\sigma$.

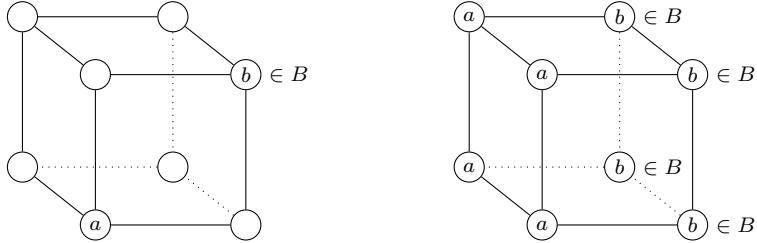
\begin{figure}[!htbp]
\begin{tikzpicture}[z={(4.85mm,-3.85mm)},every node/.style={minimum size=0.4cm,inner sep=0cm,font=\tiny},every circle node/.style={draw}]
  \node [circle] (a) at (0,0,0) {};
  \node [circle] (b) at (0,2,0) {} edge (a);
  \node [circle] (c) at (2,0,0) {} edge [dotted] (a);
  \node [circle] (d) at (2,2,0) {} edge (b) edge [dotted] (c);
  \node [circle] (e) at (0,0,2) {$a$} edge (a);
  \node [circle] (f) at (0,2,2) {} edge (b) edge (e);
  \node [circle] (g) at (2,0,2) {} edge [dotted] (c) edge (e);
  \node [circle,label={0:$~\in B$}] (h) at (2,2,2) {$b$} edge (d) edge (f) edge (g);
\end{tikzpicture}
\hspace{4em}
\begin{tikzpicture}[z={(4.85mm,-3.85mm)},every node/.style={minimum size=0.4cm,inner sep=0cm,font=\tiny},every circle node/.style={draw}]
  \node [circle] (a) at (0,0,0) {$a$};
  \node [circle] (b) at (0,2,0) {$a$} edge (a);
  \node [circle,label={0:$~\in B$}] (c) at (2,0,0) {$b$} edge [dotted] (a);
  \node [circle,label={0:$~\in B$}] (d) at (2,2,0) {$b$} edge (b) edge [dotted] (c);
  \node [circle] (e) at (0,0,2) {$a$} edge (a);
  \node [circle] (f) at (0,2,2) {$a$} edge (b) edge (e);
  \node [circle,label={0:$~\in B$}] (g) at (2,0,2) {$b$} edge [dotted] (c) edge (e);
  \node [circle,label={0:$~\in B$}] (h) at (2,2,2) {$b$} edge (d) edge (f) edge (g);
\end{tikzpicture}

\begin{center}
\caption{\sc Partial description of $t$ (left) versus total description of $t$ (right).}  \label{part_desc}
\end{center}  
\end{figure}

Split the hypercube into two parallel hyperfaces
in a way the puts the last coordinate in a different
hyperface than coordinate $\sigma$. Let $t$
be the standard generator whose coordinate value is $b$
in all coordinates of the hyperface containing
the last coordinate, and is $a$ 
in all coordinates of the hyperface containing
coordinate $\sigma$. (See Figure~\ref{part_desc}.)
This $t\in \Gamma$ satisfies
the condition in Item (3) of Claim~\ref{base_claim}.

Now we conclude the proof of the theorem.
Let $\mu = \langle \Gamma\rangle$ be the subalgebra
of $\mathbb A^{2^{c+1}}$ that is generated by
$\Gamma$. Because $\Gamma\subseteq M(\theta,\ldots,\theta)$,
we get that $\mu$ is a subalgebra
of $M(\theta,\ldots,\theta)$.
Since $M(\theta,\ldots,\theta)$ is a functional
relation from its first $2^{c+1}-1$ coordinates
to its last coordinate, $\mu$ is also a functional
relation.

Items (1) and (2) of Claim~\ref{base_claim}
guarantee that every element of $\mathbb B$
and only elements of $\mathbb B$ appear in
last coordinates of tuples of $\mu$.
If $\mathbb D\leq \mathbb A^{2^{c+1}-1}$ is the projection
of $\mu$ onto its first $2^{c+1}-1$ coordinates,
Item (3) of Claim~\ref{base_claim} guarantees
that $\mathbb D\leq \mathbb A^{2^{c+1}-1}$ is
a subdirect product representation of $\mathbb D$.
Then $\mu: \mathbb D\to \mathbb B$ is a representation
of $\mathbb B$ as a homomorphic image of a
finite subdirect power of $\mathbb A$.
This homomorphism has a right inverse $\nu:\mathbb B\to \mathbb D$
which maps $b\in\mathbb B$ to the diagonal tuple
in $\mathbb D$ with diagonal value $b$.
This represents $\mathbb B$ as a retract of finite
subdirect power of $\mathbb A$.
\end{proof}

\bigskip

{\it Proof of Theorem~\ref{main}.}
Apply Theorem~\ref{base}
in the setting where $\theta = 1_{\mathbb A}$
is the universal congruence on $\mathbb A$.  $\Box$

\section{An example}

Let $\mathbb A=\langle A; +, s, c\rangle$
be an algebra of signature $\langle 2, 1, 0\rangle$
whose universe
is $\mathbb Z_6 = \{0,1,2,3,4,5\}$ = integers modulo $6$.
The operations on $\mathbb A$ are defined by the following tables.
\[
\begin{array}{|c||cccccc|}
  \hline
+^{\mathbb A}& 0 &1 &2 &3 &4 &5\\
  \hline
  \hline
0& 0 &1 &2 &3 &4 &5\\
1&1 &2 &3 &4 &5 & 0\\
2&2 &3 &4 &5 &0 & 1\\
3&3 &4 &5 &0 &1 &2\\
4&4 &5 &0 &1 &2 &3\\
5&5 &0 &1 &2 &3 &4\\
  \hline
\end{array}
\quad\quad\quad
\begin{array}{|c||cccccc|}
  \hline
s^{\mathbb A}& 0 &1 &2 &3 &4 &5\\
  \hline
  \hline
& 0 &3 &3 &0 &3 &3\\
  \hline
\end{array}
\quad\quad\quad
c^{\mathbb A} = 3.
\]
The $+$-operation is addition modulo $6$.
We shall use the symbol $0$ to denote
the constant term $c+c$.

The subalgebra lattice
and 
the congruence lattice 
of $\mathbb A$
are depicted in Figure~\ref{SubCon}.
For each congruence we indicate the partition
into congruence classes.
We also label the congruence lattice with indices
of covering pairs.

\begin{figure}[!htbp]
\hspace{2cm}
\begin{tikzpicture}[every circle node/.style={draw,minimum size=0.2cm,inner sep=0cm,font=\tiny}]
\begin{scope}[yshift=2.5ex]
  \node [circle,label={180:$\mathbb A$}] (1) at (0,2) {};
  \node [circle,label={180:$\mathbb B$}] (0) at (0,0) {} edge (1);
\end{scope}

\begin{scope}[xshift=10em]
  \node [circle,label={0:$1$: \footnotesize $0\,1\,2\,3\,4\,5$}] (1) at (0,3) {};
  \node [label={0:\footnotesize $[1:\theta]=3$}] (l1) at (0.5,2.25) {};
  \node [circle,label={0:$\theta$: \footnotesize $0\,3 \mid 1\,4 \mid 2\,5$}] (theta) at (0,1.5) {} edge (1);
  \node [label={0:\footnotesize $[\theta:0]=2$}] (l2) at (0.5,0.75) {};
  \node [circle,label={0:$0$: \footnotesize $0 \mid 1 \mid 2 \mid 3 \mid 4 \mid 5$}] (0) at (0,0) {} edge (theta);
\end{scope}
\end{tikzpicture}

\begin{center}
\caption{\sc $\textrm{Sub}(\mathbb A)$ and $\textrm{Con}(\mathbb A)$.}\label{SubCon}  
\end{center}
\end{figure}
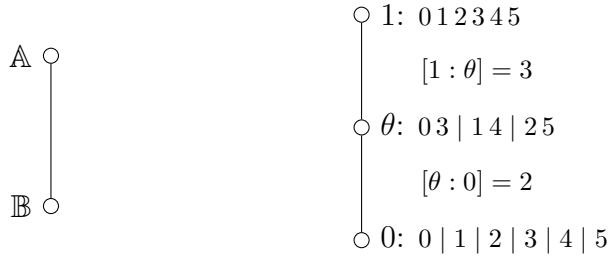
    
\noindent
The only proper subuniverse of $\mathbb A$ is
$B = \{0,3\}$, and the subalgebra supported by this
set is $\mathbb B = \langle \{0,3\}; +, s, c\rangle$
where the operations are given by the tables
\[
\begin{array}{|c||cc|}
  \hline
+^{\mathbb B}& 0 &3\\
  \hline
  \hline
0& 0 &3\\
3& 3 &0\\
  \hline
\end{array}
\quad\quad\quad
\begin{array}{|c||cc|}
  \hline
s^{\mathbb B}& 0 &3\\
  \hline
  \hline
& 0 &0\\
  \hline
\end{array}
\quad\quad\quad
c^{\mathbb B} = 3.
\]
Our goal in the section is to prove the following.

\begin{theorem} \label{example}
  \mbox{}
\begin{enumerate}
\item  $\mathbb A$ is an expansion of a finite group.
\item  $\mathbb A$ is $2$-step nilpotent as an algebraic structure,
  but it is not supernilpotent.
\item  The subalgebra $\mathbb B\leq \mathbb A$ is not
representable as a homomorphic image
of a finite subdirect power of $\mathbb A$.
\end{enumerate}
\end{theorem}

The proof of this theorem spans the rest of this section.

Item (1) of this theorem means that the term operations
of $\mathbb A$ include those of a finite group.
Those term operations are $x+y:= x +^{\mathbb A} y$,
$-x: = x +^{\mathbb A} x +^{\mathbb A} x +^{\mathbb A} x +^{\mathbb A} x = 5x$,
and $0:= c +^{\mathbb A} c$.

The assertion in Item (2) that $\mathbb A$
is $2$-step nilpotent means that $[1,1]>[1,[1,1]]=0$
where $[x,y]$ is the binary commutator.
Since, by (1), $\mathbb A$ generates a congruence
modular variety, $[x,y]=[x,y]_2$, so it does
not matter whether we are referring to
the $1$-term or $2$-term commutator.
To establish $(\leq 2)$-step nilpotence
it will be enough to check that $[1,1]\leq \theta$ and
$[1,\theta]=0$.

The algebra $\mathbb A/\theta$ is term equivalent
to the abelian group $\mathbb Z_3$, which is an abelian
algebra, and this is enough
to prove that $[1,1]\leq\theta$. To see that
$[1,\theta]=0$ it is enough to exhibit a congruence
on $\mathbb A(\theta)=\mathbf{\theta}\leq \mathbb A^2$
which has the diagonal of $A^2$ as a single class.
That congruence is:
the equivalence relation $\Delta$
on $\mathbb A(\theta)$ with exactly
two classes, the diagonal and the off-diagonal.
The reason that this is a congruence on $\mathbb A(\theta)$
is that $\Delta$ is an equivalence relation
that is compatible with the operations of the pointed
abelian group $\langle A; +, c\rangle$,
since $\theta$ is a congruence of this reduct
and this reduct is abelian. Finally,
$\Delta$ is compatible with $s$ since
$s$ maps all of $\mathbb A(\theta)$ into the diagonal,
which is a single $\Delta$-class.

We already mentioned in Corollary~\ref{main3}
that a finite supernilpotent
algebra with a Maltsev operation must factor
into a product of prime power cardinality nilpotent algebras.
Since $\mathbb A$ does not have prime power cardinality,
and its congruence lattice indicates that it has no
nontrivial direct factorization,
we derive that $\mathbb A$ cannot be supernilpotent.
It therefore cannot be abelian, hence its
nilpotence class must be exactly $2$.
This completes the proof of Item (2).

Now we focus all of our attention on the key element of
Theorem~\ref{example}, namely that
the subalgebra $\mathbb B\leq \mathbb A$ is not
representable as a homomorphic image
of a finite subdirect power of $\mathbb A$.

The language of $\mathbb A$ has $c$ and $0:=s(c) = c+c$
as constant terms, and in a given member of the
variety generated by $\mathbb A$ these terms
may or may not interpret as the same constant.
In both $\mathbb A$ and $\mathbb B$ these constants
are interpreted differently, which means that the algebras $\mathbb A$
and $\mathbb B$ fail to satisfy the identity $c=0$.
We shall prove Item~(3)
of Theorem~\ref{example} by showing that any
$2$-element algebra that is representable as a homomorphic image
of a finite subdirect power of $\mathbb A$ must satisfy
the identity
$c=0$.

Since $\mathbb A$ is an expansion of a group, we may
use concepts from group theory/ring theory.
If $\mathbb X$ is any member of the variety generated
by $\mathbb A$, then by an \emph{ideal} of $\mathbb X$
we mean a congruence class containing $0$.
(I.e., $I=0/\alpha$ for some $\alpha\in\textrm{Con}(\mathbb X)$.)
The index of one ideal in another
is computed as one would compute the index
of the underlying additive group of one
in the other. By a Sylow
$p$-subgroup of $\mathbb X$
we mean a Sylow $p$-subgroup of the underlying
additive group. We write the Sylow $p$-subgroup
of $\mathbb X$ as $V_p(\mathbb X)$.

Suppose that $\mathbb T$ is any $2$-element algebra that
is representable as a homomorphic image of some 
finite subdirect power $\mathbb D$ of $\mathbb A$.
Suppose that $I$ is an ideal
of $\mathbb D$ for which $\mathbb D/I\cong \mathbb T$.
Since $|T|=2$, it must be that $[D:I]=2$, hence
$V_3(\mathbb D)\subseteq I$. We shall argue that if
\begin{enumerate}
\item $\mathbb D\leq \mathbb A^n$ is subdirect,
\item $I$ is an ideal of $\mathbb D$, and
\item $V_3(\mathbb D)\subseteq I$, 
\end{enumerate}
then
\begin{enumerate}
\item[(4)] $c^{\mathbb D}\in I$.
\end{enumerate}
Hence $\mathbb D/I\cong \mathbb T$ satisfies
the identity $c=0$. This will show
that $\mathbb B\not\cong\mathbb T$. Since $\mathbb T$
was an arbitrary $2$-element quotient of a finite
subdirect power of $\mathbb A$, this will prove
Theorem~\ref{example}~(3).

\begin{claim}
  If $\mathbb D\leq \mathbb A^n$ is a
  representation of 
  $\mathbb D$ as a subdirect power of the
  algebra $\mathbb A$, then
  $V_3(\mathbb D)\leq V_3(\mathbb A^n)=V_3(\mathbb A)^n$ is a
  representation of the group
  $V_3(\mathbb D)$ as a subdirect power
  of the group $V_3(\mathbb A)$.
\end{claim}

Choose any $i\in\{1,\ldots,n\}$ and
any $a\in V_3(\mathbb A)$.
Use the fact that
$\mathbb D\leq \mathbb A^n$ is subdirect
to find a tuple
$\wec{d}\in\mathbb D$ such that $d_i=a$.
The underlying group of
$\mathbb A$ has an idempotent unary term operation $e(x)=4x$
for which $e(\mathbb A)=V_3(\mathbb A)$.
Now $e(\wec{d}) = (e(d_1),\ldots,e(d_n))$
has the properties that
\begin{enumerate}
\item[(i)] $e(\wec{d})_i=e(d_i)=e(a)=a$, and
\item[(ii)]
  $e(\wec{d})\in e(\mathbb D)\subseteq \mathbb D\cap e(\mathbb A^n)
= \mathbb D\cap V_3(\mathbb A^n) = V_3(\mathbb D).$
\end{enumerate}
This shows that the group $V_3(\mathbb D)$
contains a tuple $e(\wec{d})$ whose $i$-th coordinate
is $a$. Since $i\in\{1,\ldots,n\}$ and $a\in V_3(\mathbb A)$
were arbitrary, this is enough to show
that $V_3(\mathbb D)\leq V_3(\mathbb A)^n$ is subdirect.

The proof of the following claim completes the proof
of Theorem~\ref{example}~(3).

\begin{claim} \label{c}
  If $\mathbb D\leq \mathbb A^n$ is a subdirect product
  representation of $\mathbb D$, and $I$ is an ideal of $\mathbb D$
  containing $V_3(\mathbb D)$, then $c^{\mathbb D}\in I$.
\end{claim}

Partition $V_3(\mathbb D)$ into three sets,
$\{0\}, P, -P$ in any way desired, subject to
the condition $(\wec{d}\in P)\Leftrightarrow (-\wec{d}\in -P)$.
This is possible, since the group of permutations of $V_3(\mathbb D)$
consisting of the identity function $x\mapsto x$
and the negation function $x\mapsto -x$ acts on
$V_3(\mathbb D)$ in a way that partitions this set into
a single $1$-element orbit $\{0\}$ and a family
of $2$-element orbits $\{\wec{d}, -\wec{d}\}$, $\wec{d}\neq 0$,
and we may create $P$ by choosing one element from
each $2$-element orbit. We define $-P$ so that it consists
of the remaining elements.

\begin{subclaim} \label{subclaim}
  $\sum_{\wec{d}\in P} s(\wec{d}) = c^{\mathbb D}
  = (c^{\mathbb A},\ldots,c^{\mathbb A})$.
\end{subclaim}

We will show this by examining each coordinate of 
$\sum_{\wec{d}\in P} s(\wec{d})$ separately,
and showing that the result is always $c^{\mathbb A}$.
Let $i$ be an arbitrary coordinate.
Since $V_3(\mathbb D)$ is subdirect in
$V_3(\mathbb A)^n$, the projection $\pi_i$ onto the
$i$-th coordinate is a nonconstant linear functional
which maps the $\mathbb F_3$-vector space $V_3(\mathbb D)$
onto $V_3(\mathbb A)\cong \mathbb F_3$.
The size of the domain of $\pi_i$
is $|V_3(\mathbb D)|=3^k$, where $k=\dim_{\mathbb F_3}(V_3(\mathbb D))$.
The size of the (codimension-$1$) kernel of $\pi_i$
is therefore $|V_3(\mathbb D)|/3 = 3^{k-1}$,
and so
\begin{equation}\label{eqn}
|V_3(\mathbb D) \setminus \ker(\pi_i)|=3^k-3^{k-1}=2\cdot 3^{k-1}.
\end{equation}
If $\wec{d}\in \ker(\pi_i)$, then $d_i=0$, 
so $s(\wec{d})_i = s(0) = 0$.
If $\wec{d}\in V_3(\mathbb D) \setminus \ker(\pi_i)$,
then $s(\wec{d})_i = s(d_i)=c^{\mathbb A}$.
If $\wec{d}\in V_3(\mathbb D) \setminus \ker(\pi_i)$,
then $-\wec{d}\in V_3(\mathbb D) \setminus \ker(\pi_i)$,
and exactly one of $\{\wec{d},-\wec{d}\}$
belongs to $P$, so exactly half of the elements
in $V_3(\mathbb D) \setminus \ker(\pi_i)$ belong to $P$.
Hence, 
\[
\pi_i\left(\sum_{\wec{d}\in P} s(\wec{d})\right) \wrel{=}
\sum_{\wec{d}\in P} \pi_i(s(\wec{d})) \wrel{=}
\frac{1}{2}|V_3(\mathbb D) \setminus \ker(\pi_i)|\cdot c^{\mathbb A}
\;\stackrel{(\ref{eqn})}{=}\;
3^{k-1}\cdot c^{\mathbb A} \wrel{=}
c^{\mathbb A}.
\]
(The last equality uses the facts that
the additive order of $c^{\mathbb A}$ is $2$ and
$3^{k-1}$ is odd.)
Since $\pi_i(\sum_{\wec{d}\in P} s(\wec{d})) = c^{\mathbb A}$
for every $i$, we have
$\sum_{\wec{d}\in P} s(\wec{d}) = c^{\mathbb D}$. \;\; $\Box$

\bigskip

Now we complete the proof of Claim~\ref{c}.

Given that $P\subseteq V_3(\mathbb D)\subseteq I$,
we have that $\wec{d}\in P$
implies $\wec{d}\in I$, and the latter may be written as
$\wec{d}\equiv 0\pmod{I}$. Hence $s(\wec{d})\equiv s(0)=0~\pmod{I}$,
since $I$ is a congruence class.
By Subclaim~\ref{subclaim} we have
$c^{\mathbb D}=\sum_{\wec{d}\in P} s(\wec{d}) \equiv 0\pmod{I}$.
This yields $c^{\mathbb D}\in I$, which is the conclusion
to be drawn in Claim~\ref{c}.
$\Box$

\bibliographystyle{plain}

\end{document}